\documentclass{amsart} %{amsart}%{article}%{elsart}
\usepackage{graphicx}
\usepackage{amsmath}
\usepackage{sidecap}
\usepackage{amsfonts}
\usepackage{amssymb}
\usepackage{float}
\usepackage{extarrows}
\usepackage{mathrsfs}
\usepackage{booktabs}
\usepackage{verbatim}
\usepackage{hyperref}
\usepackage{esint}

\usepackage[usenames,dvipsnames]{xcolor}

\renewcommand{\baselinestretch}{1}

\newcommand{\dist}{\text{\rm dist}}

\def\bt{\begin{thm}}
\def\et{\end{thm}}
\def\bl{\begin{lem}}
\def\el{\end{lem}}
\def\bd{\begin{defi}}
\def\ed{\end{defi}}
\def\bc{\begin{cor}}
\def\ec{\end{cor}}
\def\bp{\begin{proof}}
\def\ep{\end{proof}}
\def\br{\begin{rem}}
\def\er{\end{rem}}

\def\Forall{\text{ } \forall \:}
\def\d{\, \mathrm{d}}

\def\be{\begin{equation}}
\def\ee{\end{equation}}
\def\bes{\begin{equation*}}
\def\ees{\end{equation*}}
\def\bea{\begin{equation} \begin{aligned}}
\def\eea{\end{aligned} \end{equation}}
\def\beas{\begin{equation*} \begin{aligned}}
\def\eeas{\end{aligned} \end{equation*}}
\def\ba{\begin{align}}
\def\ea{\end{align}}
\def\bas{\begin{align*}}
\def\eas{\end{align*}}

\newtheorem{thm}{Theorem}[section]
\newtheorem{lem}{Lemma}[section]
\newtheorem{defi}{Definition}[section]
\newtheorem{ex}{Example}[section]
\newtheorem{prop}[thm]{Proposition}

\newtheorem{rem}{Remark}[section]
\newtheorem{cor}{Corollary}[section]

\numberwithin{equation}{section}
\numberwithin{figure}{section}

\begin{document}
\title{Local Morrey estimate in Musielak-Orlicz-Sobolev space\footnotemark[1]\footnotemark[2]}

\author[Liu]{Duchao Liu}
\address[Duchao Liu]{School of Mathematics and Statistics, Lanzhou
University, Lanzhou 730000, P. R. China} \email{liuduchao@gmail.com, liudch@lzu.edu.cn,
Tel.: +8613893289235, fax: +8609318912481}

\author[Zhao]{Peihao Zhao}
\address[Peihao Zhao]{School of Mathematics and Statistics, Lanzhou
University, Lanzhou 730000, P. R. China} \email{zhaoph@lzu.edu.cn}

\footnotetext[2]{
Research supported by the National Natural Science Foundation of
China (NSFC 11501268 and NSFC 11471147).
}

\begin{abstract}
Under appropriate assumptions on the $N(\Omega)$-fucntion, locally uniform Morrey estimate is presented in the Musielak-Orlicz-Sobolev space. The assumptions include a new  increasing condition on the $x$-derivative of the Young complementary function of the $N(\Omega)$-fucntion. The conclusion applies to several important nonlinear examples frequently appeared in mathematical literature.
\end{abstract}
 \maketitle

%\tableofcontents

\section{Introduction}

Vast mathematical literature describes various aspects of partial differential equations related to the elliptic type operators including variable exponent, weighted, convex and double phase cases. The examples of the mentioned cases can be found in \cite{Fan5, Fan3, Marcellini89, Marcellini91, Mingione01, Mingione15, Mingione15s} and the references therein. Musielak-Orlicz-Sobolev spaces give an abstract framework of functional analysis to cover all of the above mentioned cases. 

Applications in mathematical models dealt in the framework of Musielak-Orlicz settings start from Ball's classical paper \cite{Ball76} on elasticity, investigated in the recent paper \cite{Chelminski16, Klawe16} for the model about thermo-visco-plasticity type. We refer to \cite{Gwiazda08-1, Gwiazda08-2, Gwiazda10, Gwiazda12, Wroblewska10, Wroblewska13} for the developments arising from the non-Newtonian fluids theory and some parabolic problems within the Musielak-Orlicz settings to \cite{Gwiazda14a, Gwiazda14b}.

A highly important part of the mathematical literature in general Musielak-Orlicz-Sobolev spaces gives structural conditions on regularity analysis of the space in recent years. In a recent work \cite{Ahmida18}, Ahmida and collaborators prove the density of smooth functions in the modular topology in Musielak-Orlicz-Sobolev spaces, which extends the results of Gossez \cite{Gossez82} obtained in the Orlicz-Sobolev settings. The authors impose new systematic regularity assumption on the modular function. And this allows to study the problem of density unifying and improving the known results in Orlicz-Sobolev spaces. In the paper \cite{liu_yao_18}, under some reasonable assumptions on the $N(\Omega)$-fucntion, the De Giorgi process is presented by the authors in the framework of Musielak-Orlicz-Sobolev spaces. And as the applications, the local bounded property of the minimizer for a class of the energy functional in Musielak-Orlicz-Sobolev spaces is proved. Under similar assumptions as in \cite{liu_yao_18}, the authors in \cite{Liu18} prove the H\"{o}lder continuity of the minimizers for a class of the energy functionals in Musielak-Orlicz-Sobolev spaces.

In \cite{Fan2}, Fan developed the Sobolev type inequalities in the Musielak-Orlicz-Sobolev spaces, but he did not consider the Morrey type ones. In this paper we will consider the local Morrey type inequalities and H\"{o}lder continuity of functions in Musielak-Orlicz-Sobolev spaces. We use the increasing assumption $(\widetilde{P_5})$ (see in Section \ref{Sec3}) on the $x$-derivative of the Young complementary function, which firstly appears in mathematical literature. $(\widetilde{P_5})$ can be considered as the analogous increasing condition on $x$-derivative of the Sobolev conjugate function as in $(P_5)_*$ (see in Section \ref{Sec2}, the inequality \eqref{priost}). Fortunately and importantly, both $(\widetilde{P_5})$ and $(P_5)_*$ can be implied by the analogous increasing condition $(P_5)$ on the $x$-derivative of the original $N(\Omega)$-function, see Proposition \ref{p5} in Section \ref{Sec4}.

The paper is organized as follows. In Section \ref{Sec2}, for the readers' convenience we summarize some definitions and properties about Musielak-Orlicz-Sobolev spaces. In Section \ref{Sec3}, we give the main results of this paper, including the locally uniform Morrey estimate and H\"{o}lder continuity of functions in Musielak-Orlicz-Sobolev spaces. In Section \ref{Sec4}, we discuss an important assumptions in the theorems. In Section \ref{Sec5}, we present three examples satisfying assumptions in our theorems. These examples frequently appeared in recent mathematical literatures.

\section{The Musielak-Orlicz-Sobolev Spaces}\label{Sec2}

In this section, we list some definitions and propositions related to Musielak-Orlicz-Sobolev spaces. Firstly, we give the definition of \textit{$N$-function} and \textit{generalized $N$-function} as following.

\vspace{0.3cm}

\begin{defi}
A function $A:\mathbb{R}\rightarrow[0,+\infty)$ is called an
$N$-modular function (or $N$-function), denoted by $A\in N$, if $A$ is even and convex,
$A(0)=0, 0< A(t)\in C^0$ for $t\not=0$, and the following conditions
hold
\begin{equation*}
\lim_{t\rightarrow0+}\frac{A(t)}{t}=0\text{ and }
\lim_{t\rightarrow+\infty}\frac{A(t)}{t}=+\infty.
\end{equation*}
Let $\Omega$ be a smooth domain in $\mathbb{R}^n$. A function $A:\Omega\times\mathbb{R}\rightarrow[0,+\infty)$ is
called a generalized $N$-modular function (or generalized $N$-function), denoted by $A\in N(\Omega)$, if
for each $t\in[0,+\infty)$, the function $A(\cdot,t)$ is measurable,
and for a.e. $x\in\Omega$, we have $A(x,\cdot)\in N$.
\end{defi}

\vspace{0.3cm}

Let $A\in N(\Omega)$, the Musielak-Orlicz space $L^{A}(\Omega)$ is
defined by
\begin{equation*}
\begin{aligned}
L^{A}(\Omega)&:=\bigg\{u:\,u\text{ is a measurable real function, and
}\exists\lambda>0\\
&\quad\quad\quad\quad\quad\quad\quad\quad\quad\quad\quad\quad\text{ such that
}\int_{\Omega}A\bigg(x,\frac{|u(x)|}{\lambda}\bigg)\,\mathrm{d}x<+\infty\bigg\}
\end{aligned}
\end{equation*}
with the (Luxemburg) norm
\begin{equation*}
\|u\|_{L^{A}(\Omega)}=\|u\|_{A,\Omega}=\|u\|_{A}:=\inf\bigg\{\lambda>0:\,\int_{\Omega}A\bigg(x,\frac{|u(x)|}
{\lambda}\bigg)\,\mathrm{d}x\leq1\bigg\}.
\end{equation*}

The Musielak-Sobolev space $W^{1,A}(\Omega)$ can be defined by
\begin{equation*}
W^{1,A}(\Omega):=\{u\in L^{A}(\Omega):\,|\nabla u|\in
L^{A}(\Omega)\}
\end{equation*}
with the norm
\begin{equation*}
\|u\|_{W^{1,A}(\Omega)}=\|u\|_{1,A,\Omega}=\|u\|_{1,A}:=\|u\|_A+\|\nabla u\|_{A},
\end{equation*}
where $\|\nabla u\|_{A}:=\|\,|\nabla u|\,\|_{A}$.

%$A$ is called locally integrable if $A(\cdot,t_0)\in L_{\text{loc}}^1(\Omega)$ for every $t_0>0$.

\vspace{0.3cm}

\begin{defi}\label{Musi_D}
We say that $a(x,t)$ is the Musielak derivative of $A(x,t)\in N(\Omega)$ at $t$ if
for $x\in\Omega$ and $t\geq0$, $a(x,t)$ is the right-hand derivative of
$A(x,\cdot)$ at $t$; and for $x\in\Omega$ and $t\leq0$,
$a(x,t):=-a(x,-t)$.
\end{defi}

\vspace{0.3cm}

Define $\widetilde A:\Omega\times\mathbb{R}\rightarrow[0,+\infty)$
by
\begin{equation*}
\widetilde A(x,s)=\sup_{t\in\mathbb{R}}\big(st-A(x,t)\big)\text{ for
}x\in\Omega\text{ and }s\in\mathbb{R}.
\end{equation*}
$\widetilde A$ is called the \textit{complementary function} to $A$ in the
sense of Young. It is well known that if $A\in N(\Omega)$, then
$\widetilde A\in N(\Omega)$ and $A$ is also the complementary
function to $\widetilde A$.

For $x\in\Omega$ and $s\geq0$, we denote by $\widetilde a(x,s)$ the
right-hand derivative of $\widetilde{A}(x,\cdot)$ at $s$ at the same time
define $\widetilde a(x,s)= -\widetilde a(x,-s)$ for $x\in\Omega$ and
$s\leq0$. Then for $x\in\Omega$ and $s\geq0$, we have
\begin{equation*}
\widetilde a(x,s)=\sup\{t\geq0:\, a(x,t)\leq s\}=\inf\{t>0:\,a(x,t)>s\}.
\end{equation*}

\vspace{0.3cm}

\begin{prop}[See \cite{Adams, Fan1, Musielak}]\label{Aa} 
Let $A\in N(\Omega)$. The
following assertions hold:
\begin{enumerate}
\item for any $x\in\Omega$ and
any $t\in\mathbb{R}$, 
\begin{equation*}
A(x,t)\leq a(x,t)t\leq A(x,2t);
\end{equation*}
\item for any $x\in\Omega$ and
any $t>0$,
\begin{equation*}
t<A^{-1}(x,t)\widetilde A^{-1}(x,t)\leq2t;
\end{equation*}
\item $A$ and $\widetilde A$ satisfy the Young inequality
\begin{equation*}
st\leq A(x,t)+\widetilde A(x,s) \text{ for }x\in\Omega \text{ and }
s,t\in\mathbb{R}
\end{equation*}
and the equality holds if $s=a(x,t)$ or $t=\widetilde a(x,s)$.
\end{enumerate}
\end{prop}

\vspace{0.3cm}

\begin{defi}\label{delta_2}
We say that a function $A:\Omega\times[0,+\infty)\rightarrow[0,+\infty)$
satisfies the $\Delta_2(\Omega)$ condition, denoted by $A\in
\Delta_2(\Omega)$, if there exists a positive constant $K\geq1$ such that
\begin{equation*}
A(x,2t)\leq KA(x,t)\text{ for }x\in\Omega\text{ and
}t\in[0,+\infty).
\end{equation*}
We say that a function $A:\Omega\times[0,+\infty)\rightarrow[0,+\infty)$
satisfies the $\Delta_2(\Omega)$ condition near infinity if there exist positive constants $K\geq1$ and $t_0$ such that 
\begin{equation*}
A(x,2t)\leq KA(x,t)\text{ for }x\in\Omega\text{ and
}t\geq t_0.
\end{equation*}
\end{defi}

\vspace{0.3cm}
If $A(x,t)=A(t)$ is an $N$-function in Definition \ref{delta_2}, then $A\in\Delta_2(\Omega)$ if and only if $A$ satisfies the well-known $\Delta_2$ condition defined in \cite{Adams, Trudinger_1971}.

\vspace{0.3cm}

\begin{prop} [See \cite{Fan1}]\label{int_A}
Let $A\in N(\Omega)\cap\Delta_2(\Omega)$. Then the following
assertions hold,
\begin{enumerate}
\item $L^A(\Omega)=\{u:\,u \text{ is a measurable function, and
}\int_{\Omega}A(x,|u(x)|)\,\mathrm{d}x<+\infty\}$;
\item $\int_\Omega A(x,|u|)\,\mathrm{d}x<1 \text{ (resp. } =1; >1) \Longleftrightarrow \|u\|_A<1 \text{ (resp. } =1;
>1)$, where $u\in L^A(\Omega)$;
\item $\int_\Omega A(x,|u_n|)\,\mathrm{d}x\rightarrow0 \text{ (resp. } 1; +\infty) \Longleftrightarrow \|u_n\|_A\rightarrow0 \text{ (resp. }1;
+\infty)$, where $\{u_n\}\subset L^A(\Omega)$;
\item $u_n\rightarrow u$ in $L^A(\Omega)\Longrightarrow\int_\Omega \big|A(x,|u_n|)\,\mathrm{d}x-
A(x,|u|)\big|\,\mathrm{d}x\rightarrow0$ as $n\rightarrow\infty$;
\item If $\widetilde{A}$ also satisfies $\Delta_2(\Omega)$, then
\begin{equation*}
\bigg|\int_{\Omega}u(x)v(x)\,\mathrm{d}x\bigg|\leq2\|u\|_A\|v\|_{\widetilde
A},\Forall u\in L^A(\Omega),v\in L^{\widetilde A}(\Omega);
\end{equation*}
\item $a(\cdot,|u(\cdot)|)\in L^{\widetilde A}(\Omega)$ for every $u\in
L^A(\Omega)$.
\end{enumerate}
\end{prop}

\vspace{0.3cm}

The following assumptions will be used.

\begin{enumerate}
\item[$(P_1)$] $\Omega\subset\mathbb{R}^n(n\geq2)$ is a bounded domain with the cone property,
and $A\in N(\Omega)$;
\item[$(P_2)$]
$A:\overline{\Omega}\times\mathbb{R}\rightarrow[0,+\infty)$ is
continuous and $A(x,t)\in(0,+\infty)$ for $x\in\overline{\Omega}$
and $t\in(0,+\infty)$.
\end{enumerate}

\vspace{0.3cm}

Let $A$ satisfy $(P_1)$ and $(P_2)$. Denote by $A^{-1}(x,\cdot)$ the
inverse function of $A(x,\cdot)$.  We always assume that the
following condition holds.
\begin{enumerate}
\item[$(P_3)$] $A\in N(\Omega)$ satisfies
\begin{equation}\label{0_1}
\int_0^1\frac{A^{-1}(x,t)}{t^{\frac{n+1}{n}}}\d t<+\infty,\Forall
x\in\overline\Omega.
\end{equation}
\end{enumerate}

\vspace{0.3cm}

Under assumptions $(P_1)$, $(P_2)$ and $(P_3)$, for each
$x\in\overline{\Omega}$, the function
$A(x,\cdot):[0,+\infty)\rightarrow[0,+\infty)$ is a strictly
increasing homeomorphism. Define a function $A_*^{-1}:
\overline{\Omega}\times[0,+\infty)\rightarrow[0,+\infty)$ by
\begin{equation}\label{inversA_*}
A_*^{-1}(x,s)=\int_0^s\frac{A^{-1}(x,\tau)}{\tau^{\frac{n+1}{n}}}\,\mathrm{d}\tau\text{
for }x\in\overline{\Omega} \text{ and }s\in[0,+\infty).
\end{equation}
Then under the assumption $(P_3)$, $A_*^{-1}$ is well defined, and
for each $x\in\overline{\Omega}$, $A_*^{-1}(x,\cdot)$ is strictly
increasing, $A_*^{-1}(x,\cdot)\in C^1((0,+\infty))$ and the function
$A_*^{-1}(x,\cdot)$ is concave.

Set
\begin{equation}\label{T}
T(x)=\lim_{s\rightarrow+\infty}A_*^{-1}(x,s), \Forall
x\in\overline\Omega.
\end{equation}
Then $0<T(x)\leq +\infty$. Define an even function $A_*:
\overline{\Omega}\times\mathbb{R}\rightarrow[0,+\infty)$ by
\begin{equation*}
\begin{aligned}
A_*(x,t)=\left\{ \begin{array}{ll}
          s,  & \text{ if } x\in \overline{\Omega}, |t|\in[0,T(x))\text{ and }A_*^{-1}(x,s)=|t|,\\
          +\infty,   & \text{ for } x\in \overline{\Omega} \text{ and } |t|\geq T(x).
                \end{array}\right.
\end{aligned}
\end{equation*}
Then if $A\in N(\Omega)$ and $T(x)=+\infty$ for any
$x\in\overline{\Omega}$, it is well known that $A_*\in N(\Omega)$
(see \cite{Adams}). $A_*$ is called the \textit{Sobolev conjugate function
of $A$} (see \cite{Adams} for the case of Orlicz functions).

Let $X$ be a metric space and $f:X\rightarrow(-\infty,+\infty]$ be
an extended real-valued function. For $x\in X$ with $f(x)\in
\mathbb{R}$, the continuity of $f$ at $x$ is well defined. For $x\in
X$ with $f(x)=+\infty$, we say that $f$ is continuous at $x$ if
given any $M>0$, there exists a neighborhood $U$ of $x$ such that
$f(y)>M$ for all $y\in U$. We say that
$f:X\rightarrow(-\infty,+\infty]$ is continuous on $X$ if $f$ is
continuous at every $x\in X$. Define Dom$(f)=\{x\in X :
f(x)\in\mathbb{R}\}$ and denote by $C^{0,1}_{\text{loc}}(X)$ the set of all
locally Lipschitz continuous real-valued functions defined on $X$.

\vspace{0.3cm}

The following assumptions will also be used.
\begin{enumerate}
\item[$(P_4)$] $T:\overline{\Omega}\rightarrow[0,+\infty]$ is
continuous on $\overline{\Omega}$ and $T\in C^{0,1}_{\text{loc}}(\text{Dom}(T))$;
\item[$(P_5)_*$]
$A_*\in C^{0,1}_{\text{loc}}(\text{Dom}(A_*))$ and there exist three positive constants
$\delta_*$, $C_*$ and $t_*$ with $\delta_*<\frac{1}{n}$,
$0<t_*<\min_{x\in\overline{\Omega}}T(x)$ such that
\begin{equation}\label{priost}
|\nabla_x A_*(x,t)|\leq
C_*(A_*(x,t))^{1+\delta_*},
\end{equation}
for $x\in\Omega$ and $|t|\in[t_*,T(x))$ provided $\nabla_x A_*(x,t)$
exists.
\end{enumerate}

Let $A,B\in N(\Omega)$. We say that $A\ll B$ if for any $k
> 0$,
\begin{equation*}
\lim_{t\rightarrow+\infty}\frac{A(x,kt)}{B(x,t)}=0\text{ uniformly
for }x\in\Omega.
\end{equation*}

\vspace{0.3cm}

%\begin{rem}\label{Symbol}
%Suppose that $A,B\in N(\Omega)$. Then $A\ll B\Rightarrow A\preccurlyeq B$.
%\end{rem}

%\vspace{0.3cm}

Next we present two embedding theorems for Musielak-Sobolev spaces 
developed by Fan in \cite{Fan2}.

\vspace{0.3cm}

\begin{thm}[See \cite{Fan2}, \cite{Liu_Zhao_15}]\label{imbedding}
Let $(P_1)$-$(P_4)$ and $(P_5)_*$ hold. Then
\begin{enumerate}
\item[(1)] There is a continuous imbedding $W^{1,A}(\Omega)\hookrightarrow
L^{A_*}(\Omega)$;
\item[(2)] Suppose that $B\in N(\Omega)$,
$B:\overline{\Omega}\times[0,+\infty)\rightarrow[0,+\infty)$ is
continuous, and $B(x,t)\in(0,+\infty)$ for $x\in\Omega$ and
$t\in(0,+\infty)$. If $B\ll A_*$, then there is a compact imbedding
$W^{1,A}(\Omega)\hookrightarrow\hookrightarrow L^B(\Omega)$.
\end{enumerate}
\end{thm}

\vspace{0.3cm}

Based on the definition of $T(x)$ in \eqref{T} and Theorem \ref{imbedding}, we give the following remark.

\vspace{0.3cm}

\begin{rem}\label{SoRem}
\begin{enumerate}
\item[(1)] If $T(x)=+\infty$ for any $x\in\overline\Omega$, then $(P_4)$ is automatically satisfied for $A\in N(\Omega)$ satisfying $(P_1)$, $(P_2)$ and $(P_3)$;
\item[(2)] If $T(x)<+\infty$ for any $x\in\overline\Omega$, then $W^{1,A}(\Omega)\hookrightarrow L^{\infty}(\Omega)$ for $A\in N(\Omega)$ satisfying $(P_1)$-$(P_5)$.
\end{enumerate}
\end{rem}

\section{Local Morrey estimate}\label{Sec3}

In this section, we firstly prove the locally uniform Morrey estimate and then the H\"{o}lder continuity of functions in Musielak-Orlicz-Sobolev spaces. We need the following assumption.

\vspace{0.3cm}

\begin{enumerate}
\item[$(\widetilde{P_5})$]
There exist three positive constants
$\widetilde\delta$, $\widetilde C$ and $\widetilde t$ with $\widetilde\delta<\frac{1}{n}$,
$0<\widetilde t<\min_{x\in\overline{\Omega}}T(x)$ such that
\begin{equation}\label{priowt}
|\nabla_x \widetilde A(x,t)|\leq
\widetilde C(\widetilde A(x,t))^{1+\widetilde\delta},
\end{equation}
for any $x\in\Omega$ and $t\in[\widetilde t,T(x))$ provided $\nabla_x \widetilde A(x,t)$
exists.
\end{enumerate}

\vspace{0.3cm}

\begin{thm}[Locally Uniform Continuity]\label{main}
If the $A\in N(\Omega)\cap\Delta_2(\Omega)$ satisfies $(P_1)$-$(P_4)$, $\widetilde A\in\Delta_2(\Omega)$ satisfies $(\widetilde{P_5})$, and $T(x)<+\infty$ for any $x\in\overline\Omega$, then for any given $x\in\Omega$, there exist two constants $K=K(n)>0$ and $\sigma=\sigma(x,n)>0$ such that for any $u\in W^{1,A}(\Omega)$ and any $y_1,y_2\in Q_\sigma(x)$, the following estimate holds:
\begin{equation*}
|u(y_1)-u(y_2)|\leq K\|\nabla u\|_{A,\Omega}\int_{|y_1-y_2|^{-n}}^{+\infty}\frac{A^{-1}\big(x,\tau\big)}{\tau^{\frac{n+1}{n}}}\,\mathrm{d}\tau,
\end{equation*}
where $Q_\sigma(x)\subset\Omega$ is the cube centered at $x$ with the edge length $\sigma$. 
\end{thm}

\begin{proof}
Let $Q_\sigma(x)$ denote the closed cube in $\Omega$ centered at any given $x\in\Omega$, with the edge length $\sigma<\frac{2}{\sqrt n}\dist\{x,\partial\Omega\}$. The parameter $\sigma$ will be determined later. We obtain for any $y,z\in Q_\sigma(x)$
\begin{equation*}
\begin{aligned}
|u(y)-u(z)|&=\bigg|\int_0^1\frac{\mathrm{d}}{\mathrm{d}t}u\big(y+t(z-y)\big)\,\mathrm{d}t\,\bigg|\\
&\leq|y-z|\cdot\int_0^1\big|\nabla u\big(y+t(z-y)\big)\big|\,\mathrm{d}t\\
&\leq\sigma\sqrt n\int_0^1\big|\nabla u\big(y+t(z-y)\big)\big|\,\mathrm{d}t.
\end{aligned}
\end{equation*}
It follows that
\begin{equation}\label{ineq1}
\begin{aligned}
\bigg|u(y)-\frac{1}{\sigma^n}\int_{Q_\sigma(x)}u(z)\,\mathrm{d}z\,\bigg|&=\bigg|\frac{1}{\sigma^n}\int_{Q_\sigma(x)}\big(u(y)-u(z)\big)\,\mathrm{d}z\,\bigg|\\
&\leq\frac{\sqrt n}{\sigma^{n-1}}\int_{Q_\sigma(x)}\bigg(\int_0^1\big|\nabla u\big(y+t(z-y)\big)\big|\,\mathrm{d}t\bigg)\,\mathrm{d}z\\
&=\frac{\sqrt n}{\sigma^{n-1}}\int_0^1t^{-n}\bigg(\int_{Q_{t\sigma}(x)}|\nabla u(z')|\,\mathrm{d}z'\bigg)\,\mathrm{d}t.
\end{aligned}
\end{equation}
We estimate the term $\int_{Q_{t\sigma}(x)}\nabla u(z')\,\mathrm{d}z'$ in the right hand side of inequality \eqref{ineq1}. In fact, by H\"{o}lder inequality we get
\begin{equation}\label{ineq2}
\int_{Q_{t\sigma}(x)}|\nabla u(z')|\,\mathrm{d}z'\leq2\|\nabla u\|_{A,Q_{t\sigma}(x)}\|1\|_{\widetilde A,Q_{t\sigma}(x)}
\end{equation}
To estimate $\|1\|_{\widetilde A,Q_{t\sigma}(x)}$ on the right hand side of the above inequality, from $(\widetilde{P_5})$ it is easy to see that there exist positive constants constant $C_0$, $\sigma_0=\sigma_0(x)$ and $\delta_0<\frac{1}{n}$ such that for any $s>0$ and any $y\in Q_{\sigma_0}(x)$,
\begin{equation*}
\big|\widetilde A(y,s)-\widetilde A(x,s)\big|\leq\big|(y-x)\cdot\nabla_x \widetilde A(x,s)\big|+\frac{1}{4}
\leq C_0\,|y-x|\,\big(\widetilde A(x,s)\big)^{1+\delta_0}+\frac{1}{4}.
\end{equation*} 
Then for $t\in[0,1]$, we can see that
\begin{equation*}
\begin{aligned}
&\int_{Q_{t\sigma}(x)}\widetilde A\big(y,\widetilde A^{-1}(x,\frac{t^{-n}\sigma^{-n}}{4})\big)\,\mathrm{d}y\\
\leq&\bigg|\int_{Q_{t\sigma}(x)}\widetilde A\big(x,\widetilde A^{-1}(x,\frac{t^{-n}\sigma^{-n}}{4})\big)\,\mathrm{d}y\bigg|\\
&\quad\quad+\int_{Q_{t\sigma}(x)}\bigg|\widetilde A\big(y,\widetilde A^{-1}(x,\frac{t^{-n}\sigma^{-n}}{4})\big)-\widetilde A\big(x,\widetilde A^{-1}(x,\frac{t^{-n}\sigma^{-n}}{4})\big)\bigg|\,\mathrm{d}y\\
\leq&\frac{1}{4}+C_0\int_{Q_{t\sigma}(x)}|y-x|\cdot\bigg(\widetilde A\big(x,\widetilde A^{-1}(x,\frac{t^{-n}\sigma^{-n}}{4})\big)\bigg)^{1+\delta_0}\,\mathrm{d}y+\frac{1}{4}\\
\leq&\frac{1}{2}+\frac{C_0\sqrt n}{2}\int_{Q_{t\sigma}(x)}t\sigma\cdot\frac{t^{-n-n\delta_0}\sigma^{-n-n\delta_0}}{4^{1+\delta_0}}\,\mathrm{d}y\\
=&\frac{1}{2}+\frac{C_0\sqrt n}{2\cdot4^{1+\delta_0}}\sigma^{1-n\delta_0}.
\end{aligned}
\end{equation*}
Take 
\begin{equation*}
\sigma:=\min\bigg\{\sigma_0,\bigg(\frac{4^{1+\delta_0}}{C_0\sqrt n}\bigg)^{\frac{1}{1-n\delta_0}},\frac{1}{\sqrt n}\dist\{x,\partial\Omega\}\bigg\}.
\end{equation*}
From the above inequality we can get 
\begin{equation*}
\int_{Q_{t\sigma}(x)}\widetilde A\big(y,\widetilde A^{-1}(x,\frac{t^{-n}\sigma^{-n}}{4})\big)\,\mathrm{d}y\leq\frac{1}{2}+\frac{1}{2}=1.
\end{equation*}
Then Proposition \ref{int_A}-(2) yields 
\begin{equation*}
\bigg\|\widetilde A^{-1}(x,\frac{t^{-n}\sigma^{-n}}{4})\bigg\|_{\widetilde A,Q_{t\sigma}(x)}\leq1
\end{equation*}
or equivalently
\begin{equation}\label{ineq3}
\|1\|_{\widetilde A,Q_{t\sigma}(x)}\leq\frac{1}{\widetilde A^{-1}(x,\frac{t^{-n}\sigma^{-n}}{4})}.
\end{equation}
By \eqref{ineq2} and \eqref{ineq3}, we conclude 
\begin{equation}\label{ineq4}
\int_{Q_{t\sigma}(x)}|\nabla u(z')|\,\mathrm{d}z'\leq\frac{2}{\widetilde A^{-1}(x,\frac{t^{-n}\sigma^{-n}}{4})}\|\nabla u\|_{A,Q_{t\sigma}(x)}.
\end{equation}
By Proposition \ref{Aa}-(2) for any $x\in\Omega$ and any $w\geq0$ the following inequality holds
\begin{equation*}
w\leq A^{-1}(x,w)\widetilde A^{-1}(x,w)\leq2w,
\end{equation*}
\eqref{ineq4} implies
\begin{equation*}
\int_{Q_{t\sigma}(x)}\nabla u(z')\,\mathrm{d}z'\leq8t^n\sigma^nA^{-1}\big(x,\frac{t^{-n}\sigma^{-n}}{4}\big)\|\nabla u\|_{A,Q_{t\sigma}(x)}.
\end{equation*}
Then \eqref{ineq1} and the above inequality imply
\begin{equation}\label{ineq5}
\begin{aligned}
\bigg|u(y)-\frac{1}{\sigma^n}\int_{Q_\sigma(x)}u(z)\,\mathrm{d}z\,\bigg|&\leq8\sqrt n\sigma\|\nabla u\|_{A,Q_{t\sigma}(x)}\int_0^1A^{-1}\big(x,\frac{t^{-n}\sigma^{-n}}{4}\big)\,\mathrm{d}t\\
&=\frac{8}{4^{1/n}\sqrt n}\|\nabla u\|_{A,\Omega}\int_{(4\sigma)^{-n}}^{+\infty}\frac{A^{-1}\big(x,\tau\big)}{\tau^{\frac{n+1}{n}}}\,\mathrm{d}\tau.
\end{aligned}
\end{equation}
If $y_1,y_2\in\Omega$ and $\sigma=4^{-1}|y_1-y_2|<1$, then there exists an $x\in\Omega$  and a cube $Q_\sigma(x)\subset\Omega$ with $y_1,y_2\in Q_\sigma$. And \eqref{ineq5} implies 
\begin{equation*}
|u(y_1)-u(y_2)|\leq\frac{16}{4^{1/n}\sqrt n}\|\nabla u\|_{A,\Omega}\int_{|y_1-y_2|^{-n}}^{+\infty}\frac{A^{-1}\big(x,\tau\big)}{\tau^{\frac{n+1}{n}}}\,\mathrm{d}\tau.
\end{equation*}
\end{proof}

\vspace{0.3cm}

From the proof of Theorem \ref{main}, we give the following remark.

\vspace{0.3cm}

\begin{rem}
%We only need $A,\widetilde A\in\Delta_2(\Omega)$ to insure that H\"{o}lder inequality holds. 
The following problem is open: Does the similar estimate in Theorem \ref{main} hold without $A,\widetilde A\in\Delta_2(\Omega)$?
\end{rem}
\vspace{0.3cm}

We give the definition of $\mu(\cdot)$-H\"{o}lder continuous functions on $\Omega$.

\vspace{0.3cm}

\begin{defi}
Given a positive, continuous and increasing function $\mu:\Omega\times\mathbb{R}^+\rightarrow\mathbb{R}^+$, a function $u\in C(\overline\Omega)$ is said to be $\mu(\cdot)$-H\"{o}lder continuous on $\Omega$, denoted by $u\in C^{0,\mu(\cdot)}(\overline\Omega)$, if for every $x\in\Omega$ there exists $r=r(x)$ with $0<\frac{r}{2}<\text{dist}(x,\partial\Omega)$, such that
\begin{equation*}
\sup_{y_1,y_2\in Q_r(x),
y_1\not=y_2}\frac{|u(y_1)-u(y_2)|}{\mu(x,|y_1-y_2|)}<+\infty,
\end{equation*}
where $Q_r(x)\subset\Omega$ is the cube centered at $x$ with the edge length $r$.
\end{defi}

%\red{Question: How to define a norm in $C^{0,\mu(\cdot)}(\Omega)$?}

\vspace{0.3cm}

By Remark \ref{SoRem} and Theorem \ref{main}, we obtain the following theorem.

\vspace{0.3cm}

\begin{thm}[$\mu(\cdot)$-H\"{o}lder Continuity]\label{Holder}
Suppose $A\in N(\Omega)\cap\Delta_2(\Omega)$ satisfies $(P_1)$-$(P_4)$, $(P_5)_*$, $\widetilde A\in\Delta_2(\Omega)$ satisfies $(\widetilde{P_5})$ and $T(x)<+\infty$ for any $x\in\overline\Omega$. For $s>0$ set 
\begin{equation*}
\mu(x,s):=\int_{s^{-n}}^{+\infty}\frac{A^{-1}\big(x,\tau\big)}{\tau^{\frac{n+1}{n}}}\,\mathrm{d}\tau.
\end{equation*}
If $u\in W^{1,A}(\Omega)$, then $u\in C^{0,\mu(\cdot)}(\overline\Omega)$; Moreover, for any given $x\in\Omega$, there exists two constants $K=K(n)>0$ and $\sigma=\sigma(x,n)>0$ such that for any $u\in W^{1,A}(\Omega)$ and any $y_1,y_2\in Q_\sigma(x)$, the following estimate holds:
\begin{equation*}
|u(y_1)-u(y_2)|\leq K\|u\|_{1,A,\Omega}\mu(x,|y_1-y_2|),
\end{equation*}
where $Q_\sigma(x)\subset\Omega$ is the cube centered at $x$ with the edge length $\sigma$. 
\end{thm}

\section{On the assumption $(P_5)_*$ and $(\widetilde{P_5})$}\label{Sec4}

In this section, under the assumption $(P_1)$, $(P_2)$ and $(P_3)$ the following assumption will be used.

\vspace{0.3cm}

\begin{enumerate}
\item[$(P_5)$] $A\in C^{0,1}_{\text{loc}}\big(\overline\Omega\times[0,+\infty)\big)$ and there exist positive constants $\delta_0<\frac{1}{n}$, $C_0$, and $t_0$ such that
\begin{equation}\label{priori}
|\nabla_x A(x,t)|\leq
C_0(A(x,t))^{1+\delta_0},
\end{equation}
for any $x\in\Omega$ and $t\geq t_0$ provided $\nabla_x A(x,t)$
exists.
\end{enumerate}

\vspace{0.3cm}

We will discuss the relation of the assumptions $(P_5)$, $(P_5)_*$ and $(\widetilde{P_5})$. In fact, under some reasonable assumptions, if the increasing estimation \eqref{priori} holds for $A$, then the similar increasing estimations hold for $A_*$ and $\widetilde A$, i.e., the estimations \eqref{priost} and \eqref{priowt} hold.

\vspace{0.3cm}

\begin{prop}\label{p5}
If $A$ satisfies $(P_1)$, $(P_2)$, $(P_3)$, the following conclusions hold:
\begin{enumerate}
\item[(1)] If $A$ satisfies $\Delta_2(\Omega)$ near infinity, then $(P_5)\Rightarrow(P_5)_*$;
\item[(2)] If $A\in C^{1,1}_{\text{loc}}(\overline\Omega\times[0,+\infty))$, $a(x,t):=A'_t(x,t)$ is strictly increasing on the variable $t$ for a.e. fixed $x\in\Omega$ and $\widetilde A$ satisfies $\Delta_2(\Omega)$ near infinity, then $(P_5)\Rightarrow(\widetilde{P_5})$. 
\end{enumerate}
\end{prop}

\vspace{0.3cm}

\begin{proof}
The conclusion of (1) has been proved in Proposition 3.1 of \cite{Fan2}. We will prove (2). In fact, since $a(x,s)$ is strictly increasing, we can conclude $\widetilde a=a^{-1}$. Since $A\in C^{1,1}(\overline\Omega\times[0,+\infty))$, it is easy to verify that $\nabla_xa(x,t)$ and $\nabla_x\widetilde a(x,t)$ exist for a.e. $x\in\Omega$. And by 
\begin{equation*}
\widetilde A(x,s)=s\widetilde a(x,s)-A\big(x,\widetilde a(x,s)\big)
\end{equation*}
we can get
\begin{equation*}
\begin{aligned}
\nabla_x \widetilde A(x,s)&=\nabla_x \big(s\widetilde a(x,s)-A\big(x,\widetilde a(x,s)\big)\big)\\
&=s\nabla_x\widetilde a(x,s)-\nabla_yA\big(y,\widetilde a(x,s)\big)\bigg|_{y=x}-a(x,\widetilde a(x,s))\nabla_x\widetilde a(x,s)\\
&=s\nabla_x\widetilde a(x,s)-\nabla_yA\big(y,\widetilde a(x,s)\big)\bigg|_{y=x}-s\nabla_x\widetilde a(x,s)\\
&=-\nabla_yA\big(y,\widetilde a(x,s)\big)\bigg|_{y=x}.
\end{aligned}
\end{equation*}
The above equalities, Proposition \ref{Aa} and $(P_5)$ imply that there exist positive constants $\delta_0<\frac{1}{n}$, $C_0$, and $t_0$ such that for any $x\in\Omega$ and $s\geq t_0$,
\begin{equation}\label{ineq6}
\begin{aligned}
\big|\nabla_x \widetilde A(x,s)\big|&\leq C_0\big|A(x, \widetilde a(x,s))\big|^{1+\delta_0}\\
&=C_0\big|s\widetilde a(x,s)-\widetilde A\big(x,s\big)\big|^{1+\delta_0}\\
\text{(by Proposition \ref{Aa}-(1))}&\leq C_0\big(\widetilde A(x,2s)-\widetilde A\big(x,s\big)\big)^{1+\delta_0}.
\end{aligned}
\end{equation}
Since $\widetilde A$ satisfies $\Delta_2(\Omega)$ near infinity, there exist positive constants $K_1$ and $t_1$ such that
\begin{equation*}
\widetilde A(x,2s)\leq K_1\widetilde A(x,s)
\end{equation*}
for $x\in\Omega$ and $s\geq t_0$. Then \eqref{ineq6} implies for any $x\in\Omega$ and $s\geq \max\{t_0,t_1\}$
\begin{equation*}
\big|\nabla_x \widetilde A(x,s)\big|\leq C_0\max\{K_1-1,(K_1-1)^{1+\frac{1}{n}}\}\big(\widetilde A\big(x,s\big)\big)^{1+\delta_0},
\end{equation*}
where the constant $\delta_0<\frac{1}{n}$. Then $(\widetilde{P_5})$ holds.
\end{proof}

\section{Examples}\label{Sec5}

In this section, we give several examples appeared in mathematical literature. And we prove these examples satisfies the assumptions in Theorem \ref{Holder}. In this section $\Omega\subset\mathbb{R}^n$ is a bounded domain with the cone property.

\vspace{0.3cm}

\begin{ex}[Variable exponent case]
Let $p\in
C^{1,1}(\overline\Omega)$ and $\text{sup}_{x\in{\overline\Omega}}p(x)=:p_+\geq p(y)\geq
p_-:=\text{inf}_{x\in{\overline\Omega}}p(x)>n\geq2$
for any $y\in\overline\Omega$. Define
$A:\overline\Omega\times[0,+\infty)\rightarrow[0,+\infty)$ by
\begin{equation*}
A(x,t)=\frac{t^{p(x)}}{p(x)}.
\end{equation*}
\end{ex}

\vspace{0.3cm}

It is readily checked that $A$ satisfies $(P_1)$, $(P_2)$ and
$(P_3)$. It is easy to see that $p\in C^{1,1}(\overline\Omega)$
implies $A\in C^{1,1}(\overline\Omega)$ and
\begin{equation}\label{ExAStar}
A_*^{-1}(x,s)=\frac{np(x)}{n-p(x)}(p(x))^{\frac{1}{p(x)}}s^{\frac{n-p(x)}{np(x)}}.
\end{equation}
Then $T(x)=0$ for any $x\in\overline\Omega$ and $(P_4)$ is verified.

In addition, for $x\in\Omega$,
\begin{equation*}
\nabla_x A(x,t)=t^{p(x)}\ln t\nabla p(x).
\end{equation*}
Since for any $\epsilon>0$, $\frac{\ln t}{t^\epsilon}\rightarrow0$ as $t\rightarrow+\infty$, we conclude that there
exist constants $\delta_1<\frac{1}{n}$, $c_1$ and $t_1$ such that
\begin{equation}
\bigg|\frac{\partial A(x,t)}{\partial x_j}\bigg|\leq
c_1A^{1+\delta_1}(x,t),
\end{equation}
for all $x\in\Omega$ and $t\geq t_1$. Combining
$A\in\Delta_2(\Omega)$, from Proposition \ref{p5}-(1), it is
easy to see that $(P_5)_*$ is verified.

It is readily checked that for any $x\in\overline\Omega$ and $s\geq0$,
\begin{equation*}
\widetilde A(x,s)=\frac{s^{q(x)}}{q(x)},
\end{equation*}
where $q(x)$ satisfies $\frac{1}{p(x)}+\frac{1}{q(x)}=1$. Then $\widetilde A$ satisfies $\Delta_2(\Omega)$ near infinity. Together with $a(x,t)=t^{p(x)-1}$ being strictly increasing on the variable $t$ for fixed $x\in\Omega$, we can see that $(\widetilde{P_5})$ is verified by  Proposition \ref{p5}-(2).

One can verify that in this case $\mu(\cdot)$ in Theorem \ref{Holder} is of the following form
\begin{equation*}
\mu(x,s)=\frac{np(x)}{p(x)-n}(p(x))^{\frac{1}{p(x)}}s^{1-\frac{n}{p(x)}}.
\end{equation*}

\vspace{0.3cm}

\begin{ex}[Log type case]
Let $p\in
C^{1,1}(\overline\Omega)$ and $\text{sup}_{x\in{\overline\Omega}}p(x)=:p_+\geq p(y)\geq
p_-:=\text{inf}_{x\in{\overline\Omega}}p(x)>n\geq2$
for any $y\in\overline\Omega$.  Define
$A:\overline\Omega\times[0,+\infty)\rightarrow[0,+\infty)$ by
\begin{equation*}
A(x,t)=t^{p(x)}\log(1+t),\text{ for }x\in\overline\Omega\text{ and
}t>0.
\end{equation*}
\end{ex}

It is obvious that $A$ satisfies $(P_1)$, $(P_2)$ and $(P_3)$. And for $x\in\overline\Omega$ and $t>0$ big enough, 
\begin{equation}\label{logy}
t^{p^-}\leq A(x,t)\leq t^{p^++1},
\end{equation} 
which implies that $T(x)=0$ for all $x\in\overline\Omega$. Thus $(P_4)$ is verified. Since $p\in C^{1,1}(\overline\Omega)$ and $A\in C^{1,1}(\overline\Omega\times[0,+\infty))$, by Proposition \ref{p5}-(1), $A_*\in C^{0,1}_{\text{loc}}(\overline\Omega\times[0,+\infty))$. Combining $A$ satisfies $\Delta_2(\Omega)$ near infinity, it is easy to see that the assumption $(P_5)_*$ is satisfied. By \eqref{logy} and Proposition \ref{p5}-(2), $\widetilde{(P_5)}$ can be verified. Unfortunately the explicit expression of $\mu(\cdot)$ does not exists in this case.

\vspace{0.3cm}

\begin{ex}[Double phase case]
Let $b\in
C^{1,1}(\overline\Omega)$ and $0<\text{inf}_{x\in{\overline\Omega}}\alpha(x)=:\alpha_-\leq \alpha(y)\leq
\alpha_+:=\text{sup}_{x\in{\overline\Omega}}\alpha(x)$
for any $y\in\overline\Omega$.  Define
$A:\overline\Omega\times[0,+\infty)\rightarrow[0,+\infty)$ by
\begin{equation*}
A(x,t)=t^{p}+\alpha(x)t^q,\text{ for }x\in\overline\Omega\text{ and
}t>0,
\end{equation*}
where $p$ and $q$ are constants with $q>p>n\geq2$.
\end{ex}

\vspace{0.3cm}

It is obvious that $A$ satisfies $(P_1)$, $(P_2)$ and $(P_3)$. And for $x\in\overline\Omega$ and $t>0$ big enough, 
\begin{equation}\label{logy}
A(x,t)\geq t^{p}+\alpha_-\cdot t^q,
\end{equation} 
which implies that $T(x)=0$ for all $x\in\overline\Omega$. Then $(P_4)$ is verified. Since $\alpha\in C^{1,1}(\overline\Omega)$ and $A\in C^{1,1}(\overline\Omega\times[0,+\infty))$, by Proposition \ref{p5}-(1), $A_*\in C^{0,1}_{\text{loc}}(\overline\Omega\times[0,+\infty))$. Combining $A$ satisfies $\Delta_2(\Omega)$ near infinity, the assumption $(P_5)_*$ is satisfied. By \eqref{logy} and Proposition \ref{p5}-(2), $\widetilde{(P_5)}$ can be verified. In this case, the explicit expression of $\mu(\cdot)$ does not exists.

\renewcommand{\baselinestretch}{0.1}
\bibliographystyle{plain}

%\bibliography{Ref}

\end{document}